\documentclass[10pt]{amsart}
\usepackage[margin=1in]{geometry}
\usepackage{enumerate}
\usepackage{amsopn, amsthm,amsfonts,amssymb,amscd,amsmath,enumerate,xypic, mathtools}
\usepackage[foot]{amsaddr}
\usepackage[backend=biber,style=alphabetic,sorting=ynt]{biblatex}

\usepackage{hyperref}
\usepackage{xcolor}

\bibliography{main.bib}

\usepackage{latexsym}
\usepackage{graphics}
\usepackage{color}
\newcommand{\Aut}{\text{Aut}}

\newcommand{\Z}{{\mathbb Z}}

\newcommand\C{\mathcal S}
\newcommand\A{\operatorname{Aut}_{\C}}
\newcommand\FS{\operatorname{FSym}}
\newcommand\GL[1][]{\operatorname{GL}_{#1}}

\newcommand\id{\operatorname{id}}
\newcommand\x{\overline{x}}
\newcommand{\proset}{\,\mathrel{\lower 4pt\hbox{$\scriptscriptstyle/$}
\mkern -14mu\subseteq }\,} 
 
 \newtheorem{theorem}{Theorem}[section]
  \newtheorem{corollary}[theorem]{Corollary}
 \newtheorem{lemma}[theorem]{Lemma}
 \newtheorem{proposition}[theorem]{Proposition}

 \newtheorem*{theorem*}{Theorem}

\newtheorem{remark}[theorem]{Remark}
 
 \newtheorem{definition}[theorem]{Definition}
 
 \newtheorem{example}[theorem]{Example}

\numberwithin{equation}{section}

\usepackage{amsmath}

 \makeatother 
\title{Model-theoretic $K_1$ of free modules over PIDs}
\author{Sourayan Banerjee$^1$ and Amit Kuber$^2$}  
\address{$^{1,2}$Department of Mathematics and Statistics\\ Indian Institute of Technology, Kanpur\\Uttar Pradesh-208016, India}
\email{$^1$sourayanb@iitk.ac.in, $^2$askuber@iitk.ac.in (Corresponding author)}
\keywords{$K$-theory of model-theoretic structures, $K_1$ of modules, pp-formula, principal ideal domain}
\date{}
\subjclass[2020]{03C60, 19B99, 03C07, 19D23, 19B14}
\date{} 
\begin{document}

\maketitle

\vspace{-20pt}
\begin{abstract}
Motivated by Kraji\v{c}ek and Scanlon's definition of the Grothendieck ring $K_0(M)$ of a first-order structure $M$, we introduce the definition of $K$-groups $K_n(M)$ for $n\geq0$ via Quillen's $S^{-1}S$ construction. We provide a recipe for the computation of $K_1(M_R)$, where $M_R$ is a free module over a PID $R$, subject to the knowledge of the abelianizations of the general linear groups $GL_n(R)$. As a consequence, we provide explicit computations of $K_1(M_R)$ when $R$ belongs to a large class of Euclidean domains that includes fields with at least $3$ elements and polynomial rings over fields with characteristic $0$. We also show that the algebraic $K_1$ of a PID $R$ embeds into $K_1(R_R)$.
\end{abstract}

\section{Introduction}
For a first-order structure $M$ over a language $L$, Kraji\v{c}ek and Scanlon \cite{Kra} defined the model-theoretic Grothendieck ring, denoted $K_0(M)$--this ring classifies cut-and-paste equivalence classes of definable subsets (with parameters) of finite powers of $M$ up to definable bijections. Let $\C(M)$ denote the groupoid whose objects are all definable subsets of $M^n$ for $n\geq 1$, and whose morphisms are definable bijections between them. In fact, $(\C(M),\sqcup,\emptyset,\times,\{*\})$ is a symmetric monoidal groupoid with a pairing, where $\sqcup$ denotes disjoint union, $\times$ is the Cartesian product, and $\{*\}$ is a singleton. Moreover, this association is functorial on elementary embeddings. It is easy to see that the Grothendieck ring $K_0(M)$ is exactly the ring $K_0(\C(M))$.

Given a (skeletally) small symmetric monoidal groupoid $\C$, Quillen \cite{Quillen} gave a functorial construction of the abelian groups $(K_n(\C))_{n\geq0}$, known as the $K$-theory of $\C$, which seek to classify different aspects of its objects and morphisms. When the objects of $\C$ are sets and the monoidal operation is the disjoint union, then the Grothendieck group $K_0(\C)$ classifies the isomorphism classes of objects of $\C$ up to `scissors-congruence' while the group $K_1(\C)$ classifies the automorphisms of objects of $\C$, i.e., maps that cut an object into finitely many pieces which reassemble to give the same object, in the direct limit as the objects become large with respect to $\sqcup$. Following Quillen's construction discussed above, we define the \emph{model-theoretic $K$-theory} of the first order structure $M$ by $K_n(M):=K_n(\C(M))$ for $n\geq 0$.

Associated to a unital ring $R$ is a language $ L_R:=\langle+,-,0,\{\cdot_r\mid r\in R\}\rangle$, where $\cdot_r$ is unary function symbol describing the right action of the scalar $r$. Thus, a right R-module $M_R$ can be thought of as an $ L_R$-structure. The theory $T:=Th(M_R)$ of the module $M_R$ admits partial elimination of quantifiers \cite{Baur} in terms of subgroups of $M_R$ defined using positive primitive ($pp$, for short) formulas. In fact, the theory $T$ is completely determined by the set of finite indices of pairs of $pp$-definable subgroups of $M_R$. Say that the theory $T$ is \emph{closed under products}, written $T=T^{\aleph_0}$, if given any pair $A\leq B$ of $pp$-definable subgroups of $M_R$, the index $[B:A]$ is either $1$ or $\infty$; otherwise we write $T\neq T^{\aleph_0}$.

The second author computed in \cite[Theorem~5.2.3]{Kuber1} the model-theoretic Grothendieck ring $K_0(M_R)$ of a module $M_R$, and showed that this ring is isomorphic to a certain quotient $\mathbb Z[\mathcal X]/\mathcal J$ of the integral monoid ring $\mathbb Z[\mathcal X]$, where $\mathcal X$ is the multiplicative monoid consisting of $pp$-definable subgroups of powers of $M_R$, and $\mathcal J$ is the \emph{invariants ideal} of the monoid ring encoding finite indices of pairs of $pp$-definable subgroups; $T=T^{\aleph_0}$ if and only if $\mathcal J=\{0\}$.

The main contribution of this paper is the computation of the model-theoretic group $K_1(M_R)$ for an infinite free module $M_R$ over a principal ideal domain (PID) $R$. When $R$ is a PID, the monoid $\mathcal X$ appearing in its Grothendieck ring is isomorphic to $\mathbb N$, which allows us to associate to each non-empty definable set $D$ a natural number $\dim(D)$ as its \emph{dimension}. The automorphism(=definable self-bijection) group $\Omega(D)$ of a definable set $D$ admits a finite chain of normal subgroups $(\Omega_k(D))_{k=0}^{\dim(D)}$, where $\Omega_k(D)$ is the group of automorphisms which fix all elements outside a subset of dimension at most $k$. The $pp$-definable automorphisms, which include invertible affine linear transformations, are central in determining the quotients $\Omega_{k+1}(D)/\Omega_k(D)$. Using Bass' description of $K_1(M_R)$ (Theorem \ref{K1Bass}) as the colimit $\varinjlim_{n\in\mathbb N}(\Omega(M^n))^{ab}$, we then obtain Theorem \ref{mainth}, which is the recipe for the computation of $K_1(M_R)$ subject to the knowledge of the general linear groups $GL_n(R)$ for $n\geq 1$. All the steps of this recipe are described in detail in the computation of $K_1(V_F)$ (Theorem \ref{K1FINAL}), where $V$ is an infinite vector space over an infinite field $F$, and later only the differences for the PID case are highlighted.

As a consequence of Theorem \ref{mainth}, we obtain the following explicit description of $K_1(M_R)$ from Theorems~\ref{T=T0} and \ref{TneqT0} when the ring $R$ belongs to a large class of Euclidean domains that includes fields with at least $3$ elements (Corollaries \ref{F2k} and \ref{Fpk}) and polynomial rings over fields with characteristic $0$ (Corollary \ref{polyring}).
\begin{theorem*}
Let $R$ be an Euclidean domain that contains units $u,v$ satisfying $u+v=1$. Further, suppose that $M_R$ is an infinite free right $R$-module with $T:=Th(M_R)$. Let $\mathbf{G}$ denote the set of $pp$-definable finite-index subgroups $M_R$ equipped with the subgroup relation $\leq$. If $T=T^{\aleph_0}$ or if the directed system $(\mathbf{G},\leq)^{op}$ contains a cofinal system of even-indexed subgroups of $M_R$ then  $$K_1(M_R)\cong K_1(R_R)\cong\mathbb Z_2\oplus \bigoplus_{n=1}^\infty((GL_n(R))^{ab}\oplus\mathbb Z_2)\cong \mathbb Z_2\oplus\bigoplus_{n=1}^\infty(R^\times\oplus\mathbb Z_2);$$ otherwise
$$K_1(M_R)\cong K_1(R_R)\cong\mathbb Z_2\oplus\bigoplus_{n=1}^\infty((GL_n(R))^{ab}\oplus\mathbb Z_2\oplus \mathbb{Z}_2)\cong \mathbb Z_2 \oplus \bigoplus_{n=1}^\infty(R^\times\oplus\mathbb Z_2\oplus \mathbb{Z}_2).$$
\end{theorem*}
Note in the above result that $K_1(M_R)$ depends only on the ring $R$, and not on the free module $M_R$. Even though the ring $\Z$ of integers does not satisfy the hypotheses of the above theorem, we show in Theorem \ref{ZK[x]} that $K_1(\Z_\Z)\cong\bigoplus_{n=0}^\infty\Z_2$. Our recipe to compute $K_1(M_R)$ has its limitations for explicit computations over an arbitrary PID $R$ since it heavily relies on the knowledge of the groups $(GL_n(R))^{ab}$, and unfortunately the literature in this direction is scarce. In Remarks \ref{failF2} and \ref{failMz}, we note the obstacles one would face if they were to compute $K_1(M_R)$ when $R$ is the field $F_2$ with two elements and $\Z$ respectively.  

We also give a surprising connection between algebraic $K$-theory and model-theoretic $K$-theory in the context of PIDs--no such connection is known to exist at the level of Grothendieck rings. In particular, we show in Theorem \ref{algmodK1} that the algebraic $K_1$-group over a PID $R$, $K_1^\oplus(R)$, embeds into $K_1(R_R)$.

The rest of the paper is organized as follows.
In section, \S~\ref{SemGrP}, we briefly recall the basic theory of semi-direct products, where we document the results on the abelianization of semi-direct products and certain wreath products. After setting up model-theoretic terminology, we recall the construction of the Grothendieck ring of a module thought of as a model-theoretic structure in \S~\ref{MTM}. Then in \S~\ref{KSM} we briefly recall Quillen's $K$-theory of a symmetric monoidal groupoid and use it to associate $K$-groups to a model-theoretic structure with a special emphasis on Bass' alternate description of $K_1$ (Theorem \ref{K1Bass}). Given an infinite vector space $V_F$ over an infinite field $F$, the detailed computation of $K_1(V_F)$ in \S~\ref{K1V} acts as a template for the computation of $K_1(M_R)$ in later sections. We describe in \S~\ref{K1pid} how the recipe of the previous section can be modified to compute $K_1(M_R)$ (Theorem \ref{mainth}) when $R$ is a PID and $M_R$ is an infinite free $R$-module. The goal of \S~\ref{K1ed} is to explicitly compute $K_1(M_R)$ when $R$ falls in a large class of Euclidean domains (ED). Theorems \ref{T=T0} and \ref{TneqT0} are the main results of this section. The computation of $K_1(\Z_\Z)$ (Theorem \ref{ZK[x]}) is the main goal of the short section \S~\ref{K1ed}. In the final section of the paper, \S~\ref{algK1modK1}, we discuss the connection between the algebraic $K_1$ of a PID $R$ and the group $K_1(R_R)$.

\section{Semi-direct products and abelianization}\label{SemGrP}
In this section, we recall some facts from group theory that will be used in later sections.
\begin{definition}
Let $(G,\cdotp,e)$ be a group and $N, H$ be two subgroups of $G$ with $N$ normal. We say that $G$ is the \emph{inner semi-direct product} of $N$ and $H$, written $G=H \ltimes N$ if $G = NH$ and $N \cap H =\{e\}$.

Given any two groups $H$ and $K$, and a group homomorphism $ \phi: K \rightarrow \Aut(H)$, the \emph{outer semi-direct product} of $H$ and $K$, denoted $K \ltimes_{\phi}H$, is the set $H\times K$ with multiplication defined by $(h,k)(h',k') := (h \phi(k)(h'),kk')$. We often suppress the subscript $\phi$ from the semi-direct product when the action is clear from the context.    

Given groups $K, L$ and a set $T$ on which $L$ acts, the \emph{restricted wreath product}, denoted $K\mathrm{wr}_{T}L$, is the semi-direct product $ L\ltimes (\bigoplus_{x\in T} K) $, where $l\in L$ acts on $((k_x)_{x\in T})$ as per the rule $l((k_x)_{x\in T}) := ((k_{l^{-1}x})_{x \in T})$.

For a set $T$, the finitary permutation group $\FS(T)$ on $T$ consists of permutations $\sigma \in \operatorname{Sym}(T)$ that fix all elements of $T$ outside a finite subset. Given a group $G$, the notation $G\wr\FS(T)$ denotes the restricted wreath product $G\mathrm{wr}_T\FS(T)$.
\end{definition}

For a group $G$, we denote by $G'$ its commutator subgroup $[G,G]$ and by $G^{ab}$ its abelianization $G/G'$.

The following lemma computes the abelianization of a semi-direct product.
\begin{lemma}{\label{semiab}}
Suppose $G$ is a group acting on a group $H$. Then $(G \ltimes H)^{\text{ab}} \cong  G^{\text{ab}} \times (H^{\text{ab}})_G$; here $(H^{\text{ab}})_G$ is the quotient of $H^{ab}$ by the subgroup generated by the elements of the form $h^gh^{-1}$, where $h^g$ denotes the action of $g \in G$ on $h \in H^{ab}$ induced by the action of $G$ on $H$. 
\end{lemma}

\begin{proof}[Sketch Proof]
The commutator subgroup $[G\ltimes H,G\ltimes H]$ is generated by $[H,H]\cup[G,H]\cup[G,G]$. Therefore
$(G\ltimes H)^{ab}=(G\ltimes H)/\langle[H,H]\cup[G,H]\cup[G,G]\rangle$.

Applying the relators $[H,H]$ gives $G\ltimes H^{ab}$, then applying the relators $[G,H]$ gives $G\times(H^{ab})_G$. Finally, applying $[G,G]$, we get the desired group $G^{ab}\times(H^{ab})_G$.
\end{proof}   

\begin{proposition}[\protect{\cite[\S~6.1]{BMMN}}]
 \label{commfinperm}
If $T$ is a set with at least two elements, then the finitary alternating group $\operatorname{Alt}(T)$ is the commutator subgroup of the finitary permutation group $\FS(T)$. In particular, $\FS(T)^{ab}\cong\mathbb Z_2$.
\end{proposition}

The following lemma computes the abelianization of a wreath product.
\begin{lemma}{\label{wt}}
Let $G$ be a group, $T$ be a set with at least two elements, and $\Sigma := \FS(T)$. Then $(G\wr\Sigma)^{ab} \cong G^{ab} \times \mathbb{Z}_2$.
\end{lemma}
\begin{proof}
Let $\mathbf G:=(\bigoplus_{x\in T}G_x)$, where $G_x$ is a copy of $G$ for each $x\in T$. Then $G\wr\Sigma= \Sigma \ltimes \mathbf G$, where $\Sigma$ acts on the indices of elements of $\mathbf G$. Clearly $\mathbf G^{ab}=\bigoplus_{x\in T}G^{ab}_x$. Define a map $\varepsilon:\mathbf G^{ab}\to G^{ab}$ by $(g_x)_{x\in T}\mapsto\prod_{x\in T}g_x$. It can be easily seen that $\varepsilon$ is a homomorphism.

The action of $\sigma\in\Sigma$ on $\mathbf g:=(g_x)_{x\in T}\in\mathbf G^{ab}$, denoted $\mathbf g^{\sigma}$, is given by $(g_{\sigma^{-1}x})_{x\in T}$. Let $H$ denote the subgroup of $\mathbf G^{ab}$ generated by $\{\mathbf g^{\sigma}\mathbf g^{-1}:\mathbf g\in\mathbf G^{ab},\sigma\in\Sigma\}$.

We claim that $H=\ker\varepsilon$.

Note that $\varepsilon\mathbf g=\varepsilon\mathbf g^{\sigma}$ for each $\mathbf g:=(g_x)_{x\in T}\in\mathbf G^{ab}, \sigma\in\Sigma$. Hence $H\subseteq\ker\varepsilon$.

On the other hand, consider $\mathbf g:=(g_x)_{x\in T}\in\ker\varepsilon$. Since $\Sigma$ consists only of finitary permutations of $T$, there are only finitely many $x\in T$ such that $g_x\neq 1$, say $x_1,x_2,\hdots,x_n$. We will use induction on $n$ to show that $\mathbf g\in H$.

The case when $n=0$ is trivial. If $n>0$, the identity $\prod_{i=1}^ng_{x_i}=1$ gives $n\geq2$.

Assume for induction that the result holds for all values of $n$ strictly less than $k>0$.

Suppose $n=k$. Let $\sigma$ be the transposition $(x_1,x_2)\in\Sigma$ and $\mathbf g':=(g'_x)_{x\in T}$ be the element of $\mathbf G^{ab}$ whose only non-trivial component is $g'_{x_1}=g^{-1}_{x_1}$. Let $\mathbf g'':=\mathbf g\mathbf g'((\mathbf g')^{\sigma})^{-1}$. Then $g''_{x_1}=1$ and $\varepsilon\mathbf g''=1$. The number of non-identity components of $\mathbf g''$ is strictly less than $k$ and thus, using induction hypothesis, $\mathbf g''\in H$. Therefore, $\mathbf g=\mathbf g''(\mathbf g')^{\sigma}(\mathbf g')^{-1}\in H$, thus proving the claim.

Now $(\mathbf G^{ab})_{\Sigma}=\mathbf G^{ab}/H=\mathbf G^{ab}/\ker\varepsilon=G^{ab}$. We also have $\Sigma^{ab}=\mathbb Z_2$ from Proposition \ref{commfinperm}. Thus Proposition \ref{semiab} gives $(G\wr\Sigma)^{ab}\cong(\mathbf G^{ab})_{\Sigma}\times\Sigma^{ab}\cong G^{ab}\times\mathbb Z_2$.
\end{proof}

\section{Model-theoretic Grothendieck rings of modules}\label{MTM}
\begin{definition}
Let $L$ be a language, $M$ a first-order $L$-structure with domain again denoted $M$ and $m\geq1$. Say that $A \subseteq M^m$ is \emph{definable with parameters} if there is a formula $\phi(x_1,...x_m,y_1,...,y_n)$ such that for all $a_1,...,a_m \in M$, $(a_1,...,a_m) \in A$ if and only if  $M\vDash \phi[a_1,...,a_m,b_1,...,b_n]$ for some $b_1,...,b_n \in M$. We will always assume that definable means definable with parameters from the universe.
\end{definition}
For each $m \geq 1$, let $\text{Def}(M^m)$ be the collection of all definable subsets of $M^m$, and set $\overline{\text{Def}}(M):= \bigcup_{m\geq 1} \text{Def}(M^m)$.
\begin{definition}
Say that two definable sets $A,B\in\overline{\mathrm{Def}}(M)$ are \emph{definably isomorphic} if there exists a definable bijection between them, i.e., a bijection $f:A\rightarrow B$ such that the graph $Graph(f)\in\overline{\mathrm{Def}}(M)$. Definable isomorphism is an equivalence relation on $\overline{\mathrm{Def}}(M)$ and the equivalence class of a definable set $A$ is denoted by $[A]$. We use $\widetilde{\mathrm{Def}}(M)$ 
to denote the set of all equivalence classes with respect to this relation.
\end{definition}

The assignment  $A\mapsto[A]$ defines a surjective map $[-]:\overline{\mathrm{Def}}(M)\rightarrow\widetilde{\mathrm{Def}}(M)$. We can regard $\widetilde{\mathrm{Def}}(M)$ as an $L_{ring}$-structure. In fact, it is a semiring with respect to the operations defined as follows:
\begin{itemize}
	\item $0 := [\emptyset]$;
	\item $1 := [\{*\}]$ for any singleton subset $\{*\}$ of $M$;
	\item $[A]+[B] := [A'\sqcup B']$ for $A'\in[A],B'\in[B]$ such that $A'\cap B'=\emptyset$; and
	\item $[A]\cdotp[B] := [A\times B]$.
\end{itemize}
\begin{definition}
We define the \emph{model-theoretic Grothendieck ring of the first order structure} $M$, denoted $K_0(M)$, to be the ring completion of the semiring $\widetilde{\mathrm{Def}}(M)$.
\end{definition}

Every right $R$-module $M$ is a first-order structure for the language $L_R$ of right $R$-modules, where $L_{R} := \langle +,-,0,\{\cdot_r:r \in R\} \rangle$, where each $\cdot_r$ is a unary function symbol for the scalar multiplication by $r\in R$ on the right. The right $R$-module structure $M$ will be denoted as $M_R$.

The theory of $M_R$ admits a partial elimination of quantifiers with respect to $pp$-formulas.
\begin{definition}\cite[\S~2.1]{PreBk}
A \emph{positive primitive formula} (\emph{$pp$-formula} for short) is an $L_{R}$-formula $\phi(x_1,\hdots,x_n)$ logically equivalent to one of the form
\begin{equation*}
\exists y_1\exists y_2\hdots\exists y_m\bigwedge_{i=1}^t\left(\sum_{j=1}^n x_j r_{ij}+\sum_{k=1}^m y_ks_{ik}+c_i=0\right),
\end{equation*}
where $r_{ij},s_{ik}\in R$ and the $c_i$ are parameters from $M$.

A subset $B$ of $M^m$ is \emph{$pp$-definable} if it is definable by a $pp$-formula, and a \emph{$pp$-definable function} is a function between two definable sets whose graph is $pp$-definable.
\end{definition}

\begin{lemma}
Every parameter-free $pp$-formula $\phi(\overline x)$ defines a subgroup of $M^n$, where $n$ is the length of $\overline x$. If $\phi(\overline x)$ contains parameters from $M$, then it defines either the empty set or a coset of a $pp$-definable subgroup of $M^n$. Furthermore, the conjunction of two $pp$-formulas is (logically equivalent to) a $pp$-formula.
\end{lemma}
Here is a consequence of the fundamental theorem of the model theory of modules, which is a partial quantifier elimination result due to Baur and Monk.
\begin{theorem} \cite{Baur}
For $n \geq 1$, every definable subset of $M^n$ is a finite boolean combination of $pp$-definable subsets of $M^n.$
\end{theorem}

Let $\mathcal L_n(M_R)$ (or just $\mathcal{L}_n$, if the module is clear from the context) denote the meet-semilattice of all $pp$-definable subsets of $M^n$ under intersection. Set $\overline{\mathcal {L}}(M_R):=\bigcup_{n\geq 1}\mathcal L_n(M_R)$. The notation $\overline{\mathcal{X}}(M_R)$ (or just $\overline{\mathcal X}$, if the module is clear from the context) will denote the set of \emph{colours}, i.e., $pp$-definable bijection classes of elements of $\overline{\mathcal L}(M_R)$. For $A\in\overline{\mathcal L}$, the notation $[[A]]$ will denote its $pp$-definable bijection class. The set $\overline{\mathcal{X}}^*:=\overline{\mathcal{X}}\setminus[[\emptyset]]$ of non-trivial colours is a monoid under multiplication.

\begin{definition}
Given a right $R$-module $M_R$ and subgroups $A,B\in\mathcal L_n$, define the invariant $\mathrm{Inv}(M;A,B)$ to be the index $[A:A\cap B]$ if this is finite or $\infty$ otherwise.
\end{definition}
  
\begin{definition}
The theory $T:=Th(M_R)$ of a right $R$-module $M_R$ is said to be \emph{closed under products}, written $T=T^{\aleph_0}$, if for each $n\geq 1$ and for any subgroups $A,B\in\mathcal L_n$, the invariant $\mathrm{Inv}(M;A,B)$ is either $1$ or $\infty$; otherwise we write $T\neq T^{\aleph_0}$.
\end{definition}

Nontrivial finite invariants play a central role in the computation of the model-theoretic Grothendieck ring of a module.
\begin{definition}\label{invideal}
Let $\delta_{\mathfrak A}:\overline{\mathcal X}^*\rightarrow\mathbb Z$ denote the characteristic function of a colour $\mathfrak A\in\overline{\mathcal X}^*$. Define the \emph{invariants ideal} $\mathcal J$ to be the ideal of the monoid ring $\mathbb Z[\overline{\mathcal X}^*]$ generated by the set
\begin{equation*}
    \{\delta_{[[P]]}=[P:Q]\delta_{[[Q]]}\mid P,Q\in\overline{\mathcal L},\ P\supseteq Q\supseteq\{0\},\ \mathrm{Inv}(M;P,Q)<\infty\}.
\end{equation*}
\end{definition}
We explicitly record the expression of $K_0(M_R)$ below.
\begin{theorem}\cite[Theorem~5.2.3, Corollary~5.2.11]{Kuber1}\label{FINALgeneral}
For every right $ R$-module $M_R$, we have $K_0(M_R)\cong\mathbb Z[\overline{\mathcal X}^*]/\mathcal J$, where $\mathcal{J}$ is the invariants ideal. Moreover, if $M_R\neq0$ then $K_0(M_R)\neq0$.
\end{theorem}

\section{$K$-theory for symmetric monoidal groupoids}\label{KSM}
The story of the model-theoretic $K$-theory does not end at $K_0$. We first define a symmetric monoidal category, which is nothing but the ``categorification'' of a commutative monoid. Since definable subsets of a structure under disjoint union form a symmetric monoidal category, with the help of Quillen's construction, higher model-theoretic $K$-groups are introduced in this section. We also recall Bass's definition for $K_1$ that will help us in computations in later sections.

\begin{definition}
A triple $(\C,\ast,e)$ is a \emph{symmetric monoidal category} if the category $S$ is equipped with a bifunctor $\ast:\C\times \C \to S$, and a distinguished object $e$ such that, for all objects $s,t,u\in\C$, there are natural coherent isomorphisms
\begin{displaymath}
e\ast s\cong s\cong s\ast e,\quad s\ast(t\ast u)\cong(s\ast t)\ast u,\quad s\ast t\cong t\ast s,
\end{displaymath}
that satisfy certain obvious commutative diagrams.
  A \emph{pairing} of $(\C,\ast,e)$ is a bifunctor $\otimes:\C\times \C\to \C$ such that $s\otimes e \cong e\otimes s \cong e$, and there is a natural coherent bi-distributivity law
\begin{equation*}
(s_1\ast t_1)\otimes(s_2\ast t_2)\cong(s_1\otimes s_2)\ast(s_1\otimes t_2)\ast(t_1\otimes s_2)\ast(t_1\otimes t_2).
\end{equation*}
\end{definition}
\begin{example}
The category $(\mathrm{FinSets},\sqcup,\emptyset)$ of finite sets and functions between them is a symmetric monoidal category, where $\sqcup$ denotes disjoint union, with pairing given by Cartesian product $\times$. 
\end{example}

We will only work with symmetric monoidal groupoids(categories where all morphisms are isomorphisms) in this paper. If $\C$ is a symmetric monoidal category then its subcategory $\mathrm{iso}\C$ consisting only of isomorphisms is a symmetric monoidal groupoid.

\begin{example}\label{finstructure}
For a first-order $L$-structure $M$, let $\C(M)$ denote the groupoid whose objects are $\overline{\mathrm{Def}}(M)$ and morphisms are definable bijections between definable sets. Then $(\C(M),\sqcup,\emptyset)$ is a symmetric monoidal groupoid, where $\sqcup$ is the disjoint union. Moreover, the Cartesian product of definable sets induces a pairing on this monoidal category.
\end{example}
Quillen's definition of $K$-groups of a symmetric monoidal groupoid $(\C,\ast,e)$ uses his famous $\C^{-1}\C$ construction (see \cite[Chapter~4]{Weibel} for more details).
\begin{definition}\label{KS}
If $(\C,\ast,e)$ is a skeletally small symmetric monoidal groupoid, then the $K$-\emph{theory space} $K^\ast(\C)$ of $\C$ is defined as the geometric realization of $\C^{-1}\C$. Then the $K$-groups of $\C$ are defined as $K_n^\ast(\C):=\pi_nK^\ast(\C)$.
\end{definition}
We use Quillen's definition to associate a sequence of $K$-groups with a model-theoretic structure.
\begin{definition}
For a first order structure $M$, define $K_n(M):=K_n^\sqcup(\C(M))$ for each $n\geq 0$.
\end{definition}

\begin{theorem}[ \protect{\cite[Theorem~IV.4.6]{Weibel}}]\label{pairing}
A pairing on a skeletally small symmetric monoidal groupoid determines a natural pairing $K(\C)\wedge K(\C)\to K(\C)$ of infinite loop spaces, which in turn induces bilinear products $K_p(\C)\otimes K_q(\C)\to K_{p+q}(\C)$. In particular, $K_0(\C)$ is a ring.
\end{theorem}

The Barratt-Priddy-Quillen-Segal theorem \cite[Theorem~IV.4.9.3]{Weibel}\label{BPQS} stated below is a very deep theorem connecting the $K$-theory of the apparently simple combinatorial category of finite sets and bijections with the stable homotopy theory of spheres. 
\begin{theorem}
We have $K_n(\mathrm{iso}\mathrm{FinSets})\cong\pi_n^s$ for each $n\geq0$, where $\pi^s_n$ is the $n^{th}$ stable homotopy groups of spheres.
\end{theorem}

\begin{example}
Continuing from Example \ref{finstructure}, if $M$ is a finite structure then the above theorem gives $K_n(M)\cong\pi^s_n$. In particular, $K_0(M)=\mathbb Z$ and $K_1(M)=\mathbb Z_2$.
\end{example}

Bass was the first to introduce the groups $K_1$ and $K_2$. We recall his description for $K_1$.
\begin{theorem}\cite{Bass}\label{K1Bass}
Suppose $\C$ is a symmetric monoidal groupoid whose translations are faithful i.e., for all $s,t\in S$, the translation $\mathrm{Aut}_{\C}(s)\to\mathrm{Aut}_{\C}(s\ast t)$ defined by $f\mapsto f\ast id_t$ is an injective map. Then
\begin{equation}
K_1(\C)\cong\varinjlim_{s\in \C}H_1(\mathrm{Aut}_{\C}(s);\mathbb Z),
\end{equation}
where $H_1(G;\mathbb Z)$ denotes the first integral homology group of the group $G$, which is isomorphic to $G^{ab}$.
\end{theorem}

\begin{remark}\label{ctblcofinal}
Suppose that $(\C,\ast,e)$ is a symmetric monoidal groupoid whose translations are faithful. Further suppose that $S$ has a countable sequence of objects $s_1,s_2,\hdots$ such that $s_{n+1}\cong s_n\ast a_n$ for some $a_n\in \C$, and satisfying the cofinality condition that for every $s\in \C$ there is an $s'$ and an $n$ such that $s\ast s'\cong s_n$. In this case, we can form the colimit $\mathrm{Aut}(\C):=\varinjlim_{n \in \mathbb{N}}\mathrm{Aut}_{\C}(s_n)$. Since the functor $H_1(-,\mathbb{Z})$ commutes with colimits, we obtain $K_1(\C)=H_1(\mathrm{Aut}(\C);\mathbb{Z})$.
\end{remark}
\begin{remark}\label{trfaith}
For a model-theoretic structure $M$, translations are faithful in $\C(M)$, and hence Theorem \ref{K1Bass} and the above remark can be used to compute $K_1(M)$.
\end{remark}

\section{$K_1$ of an infinite vector space over an infinite field}\label{K1V}
The goal of this section is to compute the model-theoretic $K_1$-group of a non-zero vector space $V_F$ over a field $F$ when the theory $T$ of $V_F$ is closed under products. Under this hypothesis, it can be readily seen that both $V$ and $F$ are infinite.

The theory $T$ in the language $L_F$ completely eliminates quantifiers \cite[Theorem~2.3.24]{PSL}, and hence every object of $\C:=\C(V_F)$ is a finite boolean combination of the basic definable subsets, viz., $\{0\}, V, V^2, \hdots$ and their cosets.

We start by recalling the description of the Grothendieck ring of a vector space as a special case of Theorem \ref{FINALgeneral}. Note that this result works for any field.
\begin{theorem} \cite[Theorem 4.3.1]{Perera}
The semiring $\widetilde{{Def}}(V_F)$ of definable isomorphism classes of objects of $\C(V_F)$ is isomorphic to the sub-semiring of the polynomial ring $\mathbb{Z}[X]$ consisting of polynomials with non-negative leading coefficients. As a consequence, $K_0(V_F)\cong\mathbb{Z}[X]$.
\end{theorem}

The main result of this section is the computation of $K_1(V_F)$.
\begin{theorem}\label{K1FINAL}
Suppose $V_F$ is an infinite vector space over an infinite field $F$ and $F^\times$ is the group of units in $F$. Then
\begin{equation}
K_1(V_F)\cong\mathbb Z_2\oplus\left(\bigoplus_{i=1}^\infty(F^\times\oplus\mathbb Z_2)\right).
\end{equation}
\end{theorem}

We divide the proof of this result in three steps.

\noindent{\textbf{Step I:}} In this step, we associate a ``dimension'' $\dim(f)$ to each automorphism $f$ of a definable set through its ``support'', and show that the groups of bounded-dimension automorphisms of sufficiently large definable sets are isomorphic.

First note that the assignment $\dim(\emptyset):=-\infty,\ \dim(D):=\deg([D])$ for $\emptyset\neq D\in\C$ is a well-defined dimension function on the objects of $\C$, where $[D]$ denotes the class of $D$ in the Grothendieck ring.

\begin{definition}\label{supp}
Let $D\in\C$ and $f\in\A(D)$. The \emph{support of $f$} is the (definable) set $\operatorname{Supp}(f):=\{a\in D: f(a)\neq a\}$. If $f\neq\id_{D}$ then set $\dim(f):=\dim(\operatorname{Supp}(f))$; otherwise set $\dim(\id_{D}):=-\infty$.
\end{definition}

\begin{proposition}\label{autinftydim}
For $D\in\C$, let $\Omega_m(D):=\{f\in\A(D):\dim(f)\leq m\}$ be the subgroup of $\A(D)$ of elements fixing all automorphisms of $D$ outside a subset of dimension at most $m$. If $D_1,D_2\in\C$ have dimension strictly greater than $m$, then $\Omega_m(D_1)\cong\Omega_m(D_2)$.
\end{proposition}
\begin{proof}
Since $\dim D_1,\dim D_2>m$, it is always possible to find an arrow $g:D_2\to D$ in $\C$ such that $\dim(D_1\cap D)>m$. The definable bijection $g$ induces an isomorphism between $\Omega_m(D_2)$ and $\Omega_m(D)$. Therefore, it is sufficient to prove the result when $D_1\subseteq D_2$.

For each $i=1,2$, consider the full subcategory $\C_m(D_i)$ of $\C$ containing definable subsets of $D_i$ of dimension at most $m$. The restriction of $\sqcup$ to $\C_m(D_i)$ equips it with a symmetric monoidal structure. Then $\Omega_m(D_i)\cong\mathrm{Aut}(\C_m(D_i))$, where the groups on the right hand side can be constructed as follows.

Let $S_1\subset S_2\subset\cdots$ be a sequence of objects of $\C_m(D_1)$, where $S_1$ is a copy of $V^m$ in $D_1$ and $S_{k+1}$ is obtained by adding a disjoint copy of $V^m$ to $S_k$ for each $k\geq 1$. This sequence is cofinal in $\C_m(D_1)$ and thus, using Remark \ref{ctblcofinal} 
for this sequence, we construct $\mathrm{Aut}(\C_m(D_1))$ as $\varinjlim_{k\in\mathbb N}\A(S_k)$. When $D_1\subseteq D_2$, the same colimit can be used to construct the group $\Omega_m(D_2)\cong\mathrm{Aut}(\C_m(D_2))$. Hence $\Omega_m(D_1)$ is isomorphic to $\Omega_m(D_2)$.
\end{proof}

\noindent{\textbf{Step II:}} In this step, we express groups $\Omega_n^n:=\A(V^n)$ as iterated semi-direct products, where the iterations are indexed by dimensions.

For each $0\leq m<n$, let $\Omega^n_m:=\Omega_m(V^n)$ and $\Sigma^n_m$ 
denote the finitary permutation group on a countable set of cosets of an $m$-dimensional subspace of $V^n$. If $n,p>m$, then it is easy to see that $\Sigma^n_m\cong\Sigma^p_m$. To compute groups $\Omega^n_m$, we construct a sequence $S_{m,1}\subset S_{m,2}\subset\hdots$ of objects of $\C_m(V^n)$, where $S_{m,k}$ is a disjoint union of $k$ copies of $V^m$ in $V^n$ as described in the above proposition. Note that $\Omega^n_0=\Sigma^n_0$.

We have a chain of normal subgroups of $\Omega^n_n$:
\begin{equation}
\Omega^n_0\lhd\Omega^n_1\lhd\cdots\lhd\Omega^n_{n-1}\lhd\Omega^n_n.
\end{equation}
For each $n\geq 1$, let $\Upsilon^n$ denote the subgroup of $\A(V^n)$ consisting only of definable linear (i.e., $pp$-definable) bijections. In other words, $\Upsilon^n$ is the group $\GL[n](F)\ltimes V^n$, where $\GL[n](F)$ acts on $V^n$ by matrix multiplication. Note that for each $f(\neq\id_{V^n})\in\Upsilon^n$, we have $\dim(f)=n$. The group $\Upsilon^n$ acts on $\Omega^n_{n-1}$ by conjugation and, in fact, $\Omega^n_n=\Upsilon^n\ltimes\Omega^n_{n-1}$.

For $0<m<n$, we want to find a subgroup $\Upsilon^n_m$ of $\Omega^n_m$ such that $\Omega^n_m=\Upsilon^n_m\ltimes\Omega^n_{m-1}$. To do this, we look at the construction of the colimit in Remark \ref{ctblcofinal}. Note that
\begin{equation*}
\A(S_{m,1})\cong\Omega^m_m\cong\Upsilon^m\ltimes\Omega^m_{m-1}\cong\Upsilon^m\ltimes\Omega^n_{m-1},
\end{equation*}
where the action of $\Upsilon^m$ on $\Omega^n_{m-1}$ is induced by the isomorphism $\Omega^m_{m-1}\cong\Omega^n_{m-1}$ given by Proposition \ref{autinftydim}. For similar reasons, we also have
\begin{equation*}
\A(S_{m,k})\cong(\Upsilon^m\wr\Sigma_k)\ltimes\Omega^m_{m-1}\cong(\Upsilon^m\wr\Sigma_k)\ltimes\Omega^n_{m-1},
\end{equation*}
where $\Sigma_k$ is the permutation group on $k$ elements, the group $(\Upsilon^m\wr\Sigma_k)$ acts on $\Omega^m_{m-1}$ by conjugation and permutes lower dimensional subsets of $S_{m,k}\subset V^n$. Thus
\begin{eqnarray*}
\Omega^n_m&\cong&\varinjlim_{k\in \mathbb N}\A(S_{m,k})\\
&\cong&\varinjlim_{k\in \mathbb N}\left((\Upsilon^m\wr\Sigma_k)\ltimes\Omega^n_{m-1}\right)\\
&\cong&\left(\varinjlim_{k\in \mathbb N}(\Upsilon^m\wr\Sigma_k)\right)\ltimes\Omega^n_{m-1}\\
&\cong&\left(\Upsilon^m\wr\Sigma^n_m\right)\ltimes\Omega^n_{m-1}.
\end{eqnarray*}
Define $\Upsilon^n_m:=\Upsilon^m\wr\Sigma^n_m$ which acts on $\Omega^n_{m-1}$ by conjugation. Thus each $\Omega^n_n$ is an iterated semi-direct product of certain wreath products.
\begin{eqnarray}\Omega^n_n&\cong&\Upsilon^n\ltimes\Omega^n_{n-1}\nonumber\\
&\cong&\Upsilon^n\ltimes(\Upsilon^n_{n-1}\ltimes\Omega^n_{n-2})\nonumber\\
&\cong&\Upsilon^n\ltimes(\Upsilon^n_{n-1}\ltimes(\Upsilon^n_{n-2}\ltimes\Omega^n_{n-3}))\nonumber\\
&\cong&\Upsilon^n\ltimes(\Upsilon^n_{n-1}\ltimes(\Upsilon^n_{n-2}\ltimes(\cdots(\Upsilon^n_1\ltimes\Omega^n_0)\cdots))). \label{autitsemiwr}
\end{eqnarray}

\noindent{\textbf{Step III:}} Here we compute $(\Omega^n_n)^{ab}$ using results in \S~\ref{SemGrP} so that we can compute $K_1(V_F)$ as $\varinjlim_{n\in\mathbb N}(\Omega^n_n)^{ab}$ using Theorem \ref{K1Bass} and Remark \ref{ctblcofinal} gives, thanks to Remark \ref{trfaith}.

\begin{proposition}\label{Omega}
    For each $n\geq 1$, we have
\begin{equation}
(\Omega^n_n)^{ab}\cong F^\times\oplus\bigoplus_{i=1}^{n-1}(F^\times\oplus\mathbb Z_2)\oplus\mathbb Z_2.
\end{equation}
\end{proposition}
\begin{proof}
Fix $n\geq 1$ and $0\leq m<n$. We use the presentation of $\Omega^n_n$ given in Equation \eqref{autitsemiwr}.

We know that $\Upsilon^n\cong\GL[n](F)\ltimes V^n$. Since the additive group $V^n$ is abelian, Lemma \ref{semiab} gives $(\Upsilon^n)^{ab}\cong(\GL[n](F))^{ab}\times (V^n)_{\GL[n](F)}$. For any $a(\neq 1)\in F^\times$ (which exists since the field $F$ is infinite), we have $aI_n\in\GL[n](F)$, where $I_n$ is the identity matrix. Now each $v\in V^n$ can be expressed as $(aI_n)v'-v'$ for $v'=(a-1)^{-1}v$. Thus the quotient $(V^n)_{\GL[n](F)}$ of $V^n$ is trivial which, in turn, gives $(\Upsilon^n)^{ab}\cong (\GL[n](F))^{ab}\cong F^\times$.

Recall from Proposition \ref{commfinperm} that $(\Sigma^n_m)^{ab}\cong\mathbb Z_2$. Lemma \ref{wt} applied to $\Upsilon^n_m=\Upsilon^m\wr\Sigma^n_m$ gives $(\Upsilon^n_m)^{ab}\cong(\Upsilon^m)^{ab}\times\mathbb Z_2\cong F^\times\oplus\mathbb Z_2$. The group $\Upsilon^n_m$ acts on $\Omega^n_{m-1}$ by conjugation. Recall that the action of $\Upsilon^n_m$ preserves the determinant of a matrix in $\GL[n](F)$ and the parity of a permutation. Thus repeated use of Lemma \ref{semiab} gives
\begin{eqnarray*}
(\Omega^n_n)^{ab}&\cong&(\Upsilon^n\ltimes(\Upsilon^n_{n-1}\ltimes(\cdots(\Upsilon^n_1\ltimes\Omega^n_0)\cdots)))^{ab}\\
&\cong&(\Upsilon^n)^{ab}\oplus((\Upsilon^n_{n-1}\ltimes(\cdots(\Upsilon^n_1\ltimes\Omega^n_0)\cdots))^{ab})_{\Upsilon^n}\\
&\cong&(\Upsilon^n)^{ab}\oplus((\Upsilon^n_{n-1})^{ab}\oplus((\cdots(\Upsilon^n_1\ltimes\Omega^n_0)\cdots)^{ab})_{\Upsilon^n_{n-1}})_{\Upsilon^n}\\
&\cong&(\Upsilon^n)^{ab}\oplus((\Upsilon^n_{n-1})^{ab}\oplus(\cdots ((\Upsilon^n_1)^{ab}\oplus((\Omega^n_0)^{ab})_{\Upsilon^n_1})_{\Upsilon^n_2}\cdots)_{\Upsilon^n_{n-1}})_{\Upsilon^n}\\
&\cong&F^\times\oplus((F^\times\oplus\mathbb Z_2)\oplus(\cdots((F^\times\oplus\mathbb Z_2)\oplus(\mathbb Z_2)_{\Upsilon^n_1})_{\Upsilon^n_2}\cdots)_{\Upsilon^n_{n-1}})_{\Upsilon^n}\\
&\cong&F^\times\oplus((F^\times\oplus\mathbb Z_2)\oplus(\cdots((F^\times\oplus\mathbb Z_2)\oplus\mathbb Z_2)\cdots))\\
&\cong&F^\times\oplus\bigoplus_{i=1}^{n-1}(F^\times\oplus\mathbb Z_2)\oplus\mathbb Z_2,
\end{eqnarray*}
as required.
\end{proof}
In the construction of the sequence $S_{n,1}\subset S_{n,2}\subset\cdots$ to compute $\mathrm{Aut}(\C_n(V^{n+1}))$, we can choose the copy $V^n\times\{0\}$ as $S_{n,1}$. This induces an embedding of $\Omega^n_n$ into $\Omega^{n+1}_n\lhd\Omega^{n+1}_{n+1}$. This further induces the dimension preserving inclusion of $(\Omega^n_n)^{ab}$ into $(\Omega^{n+1}_{n+1})^{ab}$. Hence
\begin{eqnarray*}
K_1(V_F)&\cong&\varinjlim_{n\in \mathbb N}(\Omega^n_n)^{ab}\\
&\cong&\varinjlim_{n\in \mathbb N}\left(F^\times\oplus\bigoplus_{i=1}^{n-1}(F^\times\oplus\mathbb Z_2)\oplus\mathbb Z_2\right)\\
&\cong&  \Z_2\oplus \left(\bigoplus_{i=1}^\infty(F^\times\oplus\mathbb Z_2)\right).
\end{eqnarray*}
This completes the proof of Theorem \ref{K1FINAL}.

\section{$K_1$ of a free module over a PID}\label{K1pid}
Throughout this section, $R$ will denote a unital PID, $M_R$ an infinite free $R$-module and $\C:=\C(M_R)$ unless stated otherwise. The main goal of this section is to compute $K_1(M_R)$. We will try to follow the proof of the vector space case (Theorem \ref{K1FINAL}), and identify the obstacles while doing so. Our goal is to prove Theorem \ref{mainth} that describes $K_1$ as a directed colimit of abelianizations of certain groups; this description works for free modules over any PID.

\noindent{\textbf{Step I:}} We defined the dimension of $D\in\C(V_F)$ as the degree of the polynomial $[D]\in\mathbb Z[X]$, but we do not have that luxury in $\C(M_R)$. Thus first we will use an alternate method to associate the notion of a non-negative integer-valued dimension for the objects of $\C(M_R)$. Recall from \S~\ref{MTM} that $\mathcal{L}_n:= \mathcal{L}_n(M_R)$ denote the meet-semilattice of all $pp$-definable subsets of $M^n$ and $\overline{\mathcal X}^*:=\overline{\mathcal X}^*(M_R)=\overline{\mathcal X}(M_R)\setminus[[\emptyset]]$ is the set of colours, i.e., the $pp$-definable isomorphism classes of non-empty $pp$-definable sets.

Let us recall some definitions and notations from \cite[\S~5.1,5.2]{Kuber1}. Let $(-)^\circ:\mathcal L_n\rightarrow \mathcal L_n$ denote the function which takes a coset $P$ to the subgroup $P^\circ:=P-p$, where $p\in P$ is any element. Denote by $\mathcal L_n^\circ$ the image of this map. The \emph{commensurability relation} on $\mathcal{L}_n$, denoted $\sim_n$, is defined by $P\sim_nQ$ if $[P^\circ:P^\circ\cap Q^\circ]+[Q^\circ:P^\circ\cap Q^\circ]<\infty$. This is, and it can be easily checked to be, an equivalence relation. The $\sim_n$-equivalence class of $P\in \mathcal L_n$ will be denoted by the corresponding bold letter $\mathbf{P}$. Let $\mathcal Y_n:=\mathcal L_n/\sim_n$ and $\overline{\mathcal Y}:=\bigcup_{n=1}^\infty\mathcal Y_n$. Given $\mathfrak A,\mathfrak B\in\overline{\mathcal X}^*$, say that $\mathfrak A\approx\mathfrak B$ if there is $\mathbf P\in\overline{\mathcal Y}$ such that $\mathbf{P}\cap\mathfrak A\neq\emptyset$ and $\mathbf{P}\cap\mathfrak B\neq\emptyset$. The relation $\approx$ is reflexive and symmetric. We use $\approx$ again to denote its transitive closure. The $\approx$-equivalence class of $\mathfrak A$ will be denoted by $\widetilde{\mathfrak A}$.
 
\begin{proposition}\label{PIDcolour}
If $M_R$ is a non-zero free $R$-module over a unital PID $R$ then $\overline{\mathcal{X}}^*(M_R)\cong \mathbb{N}$.
\end{proposition}
\begin{proof}
It suffices to check that if $\emptyset\neq D \in\mathcal L_n(M_R)$ then $[[D]] = [[M^k]]$ for some $k \leq n$. Recall from \cite[Exercise~2, p.19]{PreBk} that if $R$ is Noetherian then the set of $pp$-definable subgroups of $R_R$ is precisely the set of right ideals of $R$. Moreover, since $R$ is a PID, a $pp$-definable subgroup $D'$ of $R_R^n$ is a finitely generated submodule of the free module $R_R^n$, and hence is itself free. 

It follows from \cite[Lemma~1.2.3]{PSL} that $\mathcal L_n^\circ(R_R)\cong\mathcal L_n^\circ(M_R)$ for any non-zero $R$-module. Since every $pp$-definable subset of $M_R^n$ is $pp$-definably isomorphic to an element of $\mathcal L_n^\circ(M_R)$, the discussion in the above paragraph gives us $D \cong M^k$ for some $k\leq n$, as required.
\end{proof}
\begin{remark}
    It follows from the above that each $pp$-definable subset of $M^n$ is ($pp$-definably) isomorphic to $\prod_{i=1}^n a_i R$--a product of right ideals.
\end{remark}

Recall from the construction of $K_0(M_R)$ in \cite[Section~5.2]{Kuber1} that there is a definable-isomorphism-invariant integer-valued function $\Lambda_{\widetilde{\mathfrak{A}}}$ on the objects of $\C$ for each $\widetilde{\mathfrak{A}} \in (\overline{\mathcal{X}}^*/\approx)$  such that for each $D \in \C, \Lambda_{\Tilde{\mathfrak{A}}}(D) \neq 0$ only for finitely many values of $\widetilde{\mathfrak{A}}$. When $R$ is a PID, then it follows from Proposition \ref{PIDcolour} that the set $(\overline{\mathcal{X}}^*/\approx)$ is in bijection with $\mathbb N$. This allows us to define the dimension of a definable set.
\begin{definition}\label{dim}
    Let $M_R$ be a non-zero free $R$-module over an unital PID $R$. Define $$\dim(D) := \max\{\widetilde{\mathfrak{A}}\in(\overline{\mathcal{X}}^*/\approx) \mid\Lambda_{\widetilde{\mathfrak{A}}}(D) \neq 0\}\text{ if }\emptyset\neq D\in\C(M_R);\ \dim(\emptyset):=-\infty. $$
\end{definition}
Following the proofs of \cite[Theorems~5.2.8,5.2.9]{Kuber1} it is easy to check that $D\mapsto\dim(D)$ is monotone with respect to definable injections, and that $\dim(D_1\sqcup D_2)=\max\{\dim(D_1),\dim(D_2)\}$. The dimension of definable sets can be used to associate dimension to an automorphism of a definable set in $\C$ as in Definition \ref{supp}. In order to complete Step I, in the case when $R$ is a PID, we give a proof of the first line of the proof of Proposition \ref{autinftydim} below in Lemma \ref{defsetshift}.

Recall that a (finite) antichain in a poset is a (finite) subset in which any two distinct elements are incomparable. The set $\mathcal L_n$ is a poset with respect to inclusion. Say that a definable set $D$ is a \emph{block} if it can be written in the form $P\setminus\bigcup\beta$ for some $P\in\mathcal L_n$ and finite antichain $\beta$ in $\mathcal L_n$ satisfying $\bigcup\beta\subsetneq P$.

\begin{lemma}\label{defsetshift}
Let $D_1,D_2 \in \C(M_R)$. If $\dim(D_1)$ and $\dim(D_2)$ are both greater than $m$ then there always exists a $D \in \C(M_R)$ with a definable bijection $g : D\rightarrow D_2$ such that $\dim(D_1\cap D)>m$.
\end{lemma}
\begin{proof}
Let the dimensions of $D_1$ and  $D_2$ be $m_1$ and $m_2$ respectively. Without loss of generality, assume that $m<m_1\leq m_2$. Recall from \cite[Lemma~2.5.7]{Kuber1} that any definable set $D\subseteq M^n$ can be written as a finite disjoint union of blocks. Let $D_1=\bigsqcup_{i=1}^k B_i$ and $D_2=\bigsqcup_{j=1}^{k'} B'_j$, where $B_i$ and $B'_j$ are blocks. Without loss of generality, assume that $\dim(B_1) = m_1$ and $\dim(B'_1)=m_2$. Since each $pp$-definable set is in bijection with $M^k$ for some $k\in\mathbb N$, we have that $B_1\cong M^{m_1}\setminus(\bigcup\beta_1)$ and $B_2\cong M^{m_2}\setminus(\bigcup\beta_2)$ for appropriate antichains $\beta_1,\beta_2$.

We may assume that $\overline a\mapsto(\overline a,\overline 0)$ is an inclusion of $M^{m_1}$ into $M^{m_2}$, where $|\overline 0|=m_2-m_1$, so that $\dim(\left(M^{m_1}\setminus(\bigcup\beta_1)\right)\cap\left(M^{m_2}\setminus(\bigcup\beta_2)\right))=m_1>m$. Now we may adjust pairwise disjoint isomorphic copies $B''_j$ of $B'_j$ for $1<j\leq k'$ in such a way that $\left(M^{m_2}\setminus(\bigcup\beta_2)\right)\cap B''_j=\emptyset$. Choosing $D:=\left(M^{m_2}\setminus(\bigcup\beta_2)\right)\sqcup\bigsqcup_{j=2}^{k'}B''_j$ along with appropriate isomorphism proves the result.
\end{proof}

\noindent{\textbf{Step II:}} After modifying the definition of dimension for definable sets in the context of a PID, we need to modify the definition of the groups $\Upsilon^n:=\Upsilon^n(M_R)$ to accommodate finite-index subgroups in case $T\neq T^{\aleph_0}$, and give a suitably modified proof of $\Omega^n_n=\Upsilon^n\ltimes\Omega^n_{n-1}$. Apart from this change, the proof of Step II remains unchanged to finally yield Theorem \ref{mainth}.

For $G\in\mathcal L_n$, denote by $\mathrm{Aut}_\mathcal L(G)$ the set of $pp$-definable automorphisms of $G$. Further if $G\in\mathcal L_n^\circ$ then let $\mathrm{Aut}_{{\mathcal L}^\circ}(G)$ denote the subgroup of $\mathrm{Aut}_{{\mathcal L}}(G)$ consisting of all those automorphisms $g$ which satisfy $g(\overline 0)=\overline 0$.
\begin{proposition}\label{Aut}
If $G\in\mathcal L_n^\circ$ and $[M^n:G]$ is finite, then $\mathrm{Aut}_{\mathcal{L}}(G)\cong\mathrm{Aut}_{\mathcal{L}}(M^n) \cong (GL_n(R) \ltimes M^n)$.
\end{proposition}
\begin{proof}
Since any finite-index subgroup $G$ of $M^n$ is $pp$-definably isomorphic to $M^n$ we get the first isomorphism. Let $\pi_1$ and $\pi_2$ be the projection maps of $M^{2n}$ of the first and last $n$ coordinates respectively.

Recall that $\mathcal L_n^\circ(R_R)\cong\mathcal L_n^\circ(M_R)$ \cite[Lemma~1.2.3]{PSL} for any non-zero free $R$-module $M_R$. Thus, we get the isomorphism $\mathrm{Aut}_{{\mathcal L}^\circ}(M^n) \cong \mathrm{Aut}_{{\mathcal L}^\circ}(R^n)$. Now $\mathrm{Aut}_{\mathcal{L}^\circ}(M^n)$ acts on the group $M^n$ naturally, and hence $\mathrm{Aut}_{\mathcal{L}}(M^n) \cong \mathrm{Aut}_{{\mathcal L}^\circ}(M^n) \ltimes M^n$ via the map $g \mapsto (g',(g')^{-1}(g(\overline 0))$, where $g'(\overline x):=g(x)-g(\overline 0)$. Combining the previous two statements, we get that $\mathrm{Aut}_{\mathcal{L}}(M^n)\cong \mathrm{Aut}_{{\mathcal L}^\circ}(R^n) \ltimes M^n$. Thus it remains to show that $\mathrm{Aut}_{{\mathcal L}^\circ}(R^n) \cong GL_n(R)$. Since multiplication by an invertible matrix is a $pp$-definable map that preserves $\overline 0$, we only need to show that each element in $\mathrm{Aut}_{{\mathcal L}^\circ}(R^n)$ arises in this way.

Let $f\in\mathrm{Aut}_{{\mathcal L}^\circ}(R^n)$ and $\phi(\overline{x},\overline{y})$ be the $pp$-formula defining the subgroup $D\in \mathcal L_{2n}(R_R)$ that is the graph of $f$. Since $R$ is a PID, thanks to \cite[Theorem~2.3.19]{PSL}, $D$ is a finitely generated submodule of $R^{2n},$ and hence is free. Since $D$ is the graph of an automorphism, its rank equals the rank of its domain, which equals $n$. Let $\{\overline{z}_1,\overline z_2,\hdots,\overline z_n\} \subseteq D$ be a basis of $D$. Using the entries of the basis elements, we can construct an $(n\times 2n)$-matrix $[A\ B]$ over $R$ such that for any $\overline a\in \pi_1(D)= R^n$ there is a unique $\overline w \in R^n$ and $\overline w[A\ B]=(\overline a,f(\overline a))$, where $A$ and $B$ are respectively the left and right $n\times n$ blocks of the $n\times 2n$ matrix. Since $\pi_1(D)=\pi_2(D)=R^n$, we conclude that $A,B\in GL_n(R)$. Then by a suitable change of basis, we may assume that $\overline w=\overline a$ and $A=I$ is the identity matrix, thereby transforming the block matrix to $[I\ B']$ for $B'\in GL_n(R)$. Thus the required matrix presenting $f$ is $-B'$. This completes the proof. 
\end{proof}

Recall that $\Upsilon^n(V_F)=\mathrm{Aut}_{\mathcal L}(V^n_F)$. However, when $R$ is a PID, the possible existence of finite-index subgroups of $M_R$ requires us to reinterpret the group $\Upsilon^n(M_R)$.

\begin{definition}\label{Upsilon}
Say that $f \in \A(M^n)$ is an $n$-\emph{automorphism} if the formula $\phi(\overline{x},\overline{y})$ defining the graph of $f$ is logically equivalent to $\bigvee_{i=1}^m \phi_i(\overline{x},\overline{y})$, where each $\phi_i$ is $pp$-definable and $\dim(\phi_i(M_R)) = n.$ Let $\Upsilon^n:=\Upsilon^n(M_R)$ be the set of all $n$-automorphisms $f\in\A(M^n)$.
\end{definition}
Since a projection of $pp$-definable sets is again so, given an $n$-automorphism $f$ of $M^n$, the set $H(f):=\bigcap_{i=1}^m(\pi_1(\phi_i(M_R^n)))^\circ $ is a $pp$-definable subgroup of $M^n$ of finite-index, say $k$, where $\pi_1$ is the projection of $M^{2n}$ onto the first $n$ coordinates. Thus, $\phi(\overline{x},\overline{y})$ is logically equivalent to $\bigvee_{H(f)+\overline p_j\in M^n/H(f)}\phi_j(\overline{x},\overline{y})$, where $\pi_1\phi_j(M^n) = H(f)+\overline{p}_j$. So in other words, $f\in\A(M^n)$ is an $n$-automorphism if and only if there is a finite-index subgroup $H(f) \in \mathcal L_n$ and a permutation $\sigma_f$ of $M^n/H(f)$ such that for every $H(f)+q\in M^n/H(f)$, the restriction of $f\mid_{H(f)+q}:H(f)+q \rightarrow \sigma_f(H(f)+q)$ is a $pp$-definable bijection.  Since the intersection of two finite-index subgroups of an abelian group is again so, we get the following.
\begin{proposition}\label{sbgpupsi}
The set $\Upsilon^n$ is a subgroup of $\A(M^n)$.
\end{proposition}

Propositions \ref{sbgpupsi} and \ref{Aut} together with the discussion above yields the following description of $\Upsilon^n$ as a colimit of a system of groups.
\begin{proposition}\label{Upsi}
    Let $\mathbf G_n$ be the set of all $pp$-definable finite-index subgroups of $M^n_R$. Then $(\mathbf G_n,\leq)$ is a poset under the subgroup relation in such a way that $(\mathbf G_n,\leq)^{op}$ is a directed set. Then 
    \begin{displaymath}
    \Upsilon^n(M_R) \cong \varinjlim_{G \in (\mathbf G_n,\leq)^{op}} (\mathrm{Aut}_{\mathcal{L}}(G)\wr \FS(M^n/G)) \cong\varinjlim_{G \in (\mathbf G_n,\leq)^{op}} \left((GL_n(R)\ltimes M^n)\wr \FS(M^n/G)\right).  
    \end{displaymath}
\end{proposition}

Let us explain the directed colimit in the above result through the example of $\Upsilon^1(\mathbb Z_\mathbb Z)$: $$\Upsilon^1(\mathbb Z_{\mathbb Z}) \cong \varinjlim_{G \in (\mathbf G_1(\mathbb Z_\mathbb Z),\leq)^{op}}  \left(\mathrm{Aut}_{\mathcal{L}}(G)\wr \FS(\mathbb Z/G)\right) \cong \varinjlim_{n \in (\mathbb N^*, \mid)}  \left(\mathrm{Aut}_{\mathcal{L}}(n\mathbb Z)\wr \FS(\mathbb Z_n)\right),$$ where $\mathbb N^*:= \mathbb N \setminus \{0\}$ and $\mid$ is the divisibility relation.

Suppose $n\mid m$ in $\mathbb{N}^*$. Then there is a natural quotient map ${\theta}:\mathbb Z_m \rightarrow \mathbb Z_n$. We explicitly describe the embedding of $\left(\mathrm{Aut}_{\mathcal{L}}(n\mathbb Z)\wr \FS(\mathbb Z_n)\right) \cong \FS(\mathbb Z_n) \ltimes  \bigoplus_{l\in \mathbb{Z}_n}(\mathbb{Z}_2 \ltimes  n\mathbb{Z}) $ into $\left(\mathrm{Aut}_{\mathcal{L}}(m\mathbb Z)\wr \FS(\mathbb Z_m)\right) \cong \FS(\mathbb Z_m) \ltimes  \bigoplus_{l'\in \mathbb{Z}_m}(\mathbb{Z}_2 \ltimes  m\mathbb{Z})$. Recall from Proposition \ref{Aut} that $\mathrm{Aut}_\mathcal L(n\mathbb Z)\cong GL_1(\mathbb{Z})\ltimes n\mathbb Z$. Clearly $GL_1(\mathbb Z)\cong\mathbb Z_2$.

A natural way to think about the required embedding is component-wise. Define the embedding $\alpha_1: \FS(\mathbb Z_n) \hookrightarrow \FS(\mathbb{Z}_m)$ by $\alpha_1(\sigma)(t) := (t+\sigma(\theta(t))-\theta(t))\;(\text{mod}\;m)$ for $\sigma\in\FS(\mathbb Z_n)$. Also define the embedding $\alpha_2:  \bigoplus_{l\in \mathbb{Z}_n}(\mathbb{Z}_2 \ltimes  n\mathbb{Z}) \hookrightarrow  \bigoplus_{l'\in \mathbb{Z}_m}(\mathbb{Z}_2 \ltimes  m\mathbb{Z})$ by $$\alpha_2((\zeta_l,h_l)_{l \in \mathbb{Z}_n}) := ((\zeta_{\theta(l')},h_{\theta(l')})_{l'\in \mathbb{Z}_m}).$$ It is routine to verify that the map $(\alpha_1,\alpha_2)$ is indeed the required embedding. The description of the maps $\alpha_1,\alpha_2$ looks complicated but it is not as the next examples demonstrate. 

\begin{example} Let $n=2,m=4$ and $\sigma\in\FS(\mathbb Z_2)$ is the only non-trivial permutation. Then $\alpha_1(\sigma)$ is given by assignment $0\mapsto 1,1\mapsto 0,2\mapsto3$ and $3\mapsto2.$  
\end{example}

\begin{example}
Let $n=3$ and $m=6$. Then $((\zeta,h_0),( \zeta', h_1),( \zeta'',h_2)) \in \bigoplus_{l\in \mathbb{Z}_3} (\mathbb{Z}_2 \ltimes 3\mathbb{Z})$ maps to $\alpha_2$ is $((\zeta,h_0),( \zeta', h_1),( \zeta'',h_2),(\zeta,h_0),( \zeta', h_1),( \zeta'',h_2)) \in \bigoplus_{l'\in \mathbb{Z}_6} (  \mathbb{Z}_2 \ltimes 6\mathbb{Z})$ under $\alpha_2$.
\end{example}

Let $G\leq G'$ in $(\mathbf G_1(\mathbb Z_\mathbb Z),\leq)$. If $[G':G]$ is even then for any $\sigma \in \FS(\mathbb Z/G')$, the permutation $\alpha_1(\sigma)$ in $\FS(\mathbb Z/ G)$ is an even permutation.  The same statement also holds when $\mathbf G_1(\mathbb Z_{\mathbb Z})$ is replaced by $\mathbf G_1(M_R)$ for any $k \geq 1$. After this discussion, the following is not hard to prove.
\begin{proposition}\label{evsbg}
The following are equivalent for any free $R$-module $M_R$.
\begin{enumerate}
    \item There is a cofinal system of even-indexed subgroups in $(\mathbf{G}_1(M_R),\leq)^{op}$.
    \item There is a cofinal system of even-indexed subgroups in $(\mathbf{G}_k(M_R),\leq)^{op}$.
    \item Each permutation is eventually even in the directed colimit in Proposition \ref{Aut}.
\end{enumerate}
\end{proposition}

Let us use the notations as well as definitions of groups $\Omega^n_m$ from Step II of the proof of Theorem \ref{K1FINAL}. The group $\Upsilon^n$ clearly acts on $\Omega^n_{n-1}$ by conjugation.
\begin{lemma}
For each $n\geq1$, we have $\Omega^n_n\cong\Upsilon^n\ltimes\Omega^n_{n-1}$.
\end{lemma}
\begin{proof}\label{UpOmcong}
Let $f \in \A(M^n)$, $\phi(\overline{x},\overline{y})$ be an $L_R$-formula defining the graph $D\subseteq M^{2n}$ of $f$. Recall from \cite[Lemma~2.5.7]{Kuber1} that $D = \bigsqcup_{i=1}^{k_1+k_2}B_i$, where each $B_i$ is a block and $\dim(B_i)=n$ if and only if $i\leq k_1$. Further partitioning into smaller blocks if necessary, we may assume for each $i$ that $B_i= P_i\setminus \bigcup \beta_i$ for some $P_i \in \mathcal{L}_n$ and an antichain $\beta_i$ in $\mathcal{L}_n$ satisfying $\dim(\bigcup \beta_i) < \dim(P_i)$. Then $P_0:= \bigcap_{i=1}^{k_1} (\pi_1(P_i))^\circ$ is a finite-index subgroup of $M^n$. Let $\{P_0^j\mid 0\leq j \leq k-1\}$ be the set of all cosets of $M^n/P_0$. For each $0\leq j<k$, let $B_i^j\subseteq B_i$ be the block satisfying $\pi_1(B_i^j)=\pi_1(B_i)\cap P_0^j$ so that $B_i=\bigsqcup_{0\leq j<k}B_i^j$ for each $1\leq i\leq k_1$. Let $B_i^j=P_i^j\setminus\bigcup\beta_i^j$ for some $P_i^j\in\mathcal L_{2n}$ and an antichain $\beta_i^j$ in $\mathcal L_{2n}$ satisfying $\dim(\bigcup\beta_i^j)<n=\dim(P_i^j)$. Let $\phi_i^j(\overline x,\overline y)$ be the $pp$-definable formula defining $P_i^j$. Then $\bigvee_{i=1}^{k_1}\bigvee_{j=0}^{k-1}\phi_i^j(\overline{x},\overline{y})$ is a formula defining the graph of an $n$-automorphism, say $g$, of $M^n$. We have also ensured $g^{-1}f\in\Omega^n_{n-1}$ to complete the proof.
 \end{proof}

In view of the above lemma, the remaining arguments of Step II of the proof of Theorem \ref{K1FINAL} go through to give the following result for a free infinite right $R$-module $M_R$ over a PID $R$.
\begin{theorem}\label{mainth}
    Suppose $R$ is a PID and $M_R$ is an infinite free right $R$-module. Then using the notations introduced in Step II of the proof of Theorem \ref{K1FINAL}, and with $\Upsilon^n$ as in Definition \ref{Upsilon}, we have $$K_1(M_R) = \varinjlim_{n\in \mathbb N}(\Omega_n^n)^{ab}.$$ 
\end{theorem}

\section{Computation of $K_1$ of free modules over certain Euclidean domains}\label{K1ed}
In general, given a PID $R$, exact computation of $(GL_2(R))^{ab}$ is not known, and thus the computation of $K_1(M_R)$ using the recipe in Theorem \ref{mainth} is not possible for all PIDs. However, under a mild condition on an ED $R$, we compute $K_1$ for free $R$-modules in Theorems \ref{T=T0} and \ref{TneqT0}. As a consequence, we compute $K_1(M_{F[X]})$ for a field $F$ with $\mathrm{char}(F)=0$ (Corollary \ref{polyring}) as well as $K_1(V_F)$ for a field $F$ with at least $3$ elements (Corollaries \ref{F2k} and \ref{Fpk}).   

For a commutative unital ring $R$, the notation $SL_n(R)$ denotes the group of special linear group, i.e., $n\times n$ matrices with determinant $1$ so that $GL_n(R)/SL_n(R)\cong R^\times$ for each $n\geq1$. The notation $E_n(R)$ denote the subgroup of $GL_n(R)$ generated by $n\times n$ elementary matrices \cite[Definition~III.1.2]{Weibel}.

First we compute the abelianization of the action of $GL_n(R)$ on $M^n$ by multiplication.
\begin{lemma}\label{quotfree}
Suppose $R$ is any commutative ring with unity and $M_R$ is a right $R$-module. Then for each $n\geq2$ we have $(GL_n(R) \ltimes M^n)^{ab} \cong (GL_n(R))^{ab}$. Moreover, if the multiplicative identity $1$ in $R$ can be written as a sum of two units then the conclusion also holds true for $n=1$.
\end{lemma}
\begin{proof}
We prove the lemma when $M_R = R_R$; the proof for $M_R$ follows verbatim. Thanks to Lemma \ref{semiab}, we have $(GL_n(R) \ltimes R^n)^{ab} \cong (GL_n(R))^{ab} \times R^n_{(GL_n(R))}$, so it is enough to show that $R^n_{(GL_n(R))}$ vanishes. By definition, $R^n_{(GL_n(R))} = R^n/\langle \overline{x}A - \overline{x}\mid A \in GL_n(R),\ \overline{x} \in R^n\rangle$.  Let $\overline{x} = (x_1,x_2,....,x_n) \in R^n$ and for $i \neq j$; $E_{ij} \in E_n(R)$ be such that $E_{ij} = I+e_{ij}$,  where $e_{ij} \in M_n(R)$ and every entry of the matrix $e_{ij}$ is $0$ apart from the $ij^{th}$ entry, which equals $1$. If $E_{ij} \in \{E_{1n}\}\cup\{E_{k{k-1}}\mid 2\leq k\leq n\}$ then $\x E_{ij} - \x = (0,\hdots,x_i,\hdots,0)$, where $x_i$ appears in the $j^{th}$ place. Since $x_i \in R$ is arbitrary we conclude that $\langle \overline{x}A - \overline{x}\mid A \in GL_n(R),\ \overline{x} \in R^n\rangle = R^n$.   
\end{proof}
\begin{remark}\label{chad}
The above lemma need not hold for $n=1$ if $1$ cannot be written as a sum of two units, e.g., if $R=M=\mathbb{Z}$ then $(GL_{1}(\mathbb Z) \ltimes \mathbb{Z})^{ab}=(\mathbb{Z}_2 \ltimes \mathbb{Z})^{ab} \cong \mathbb{Z}_2 \oplus \mathbb{Z}_2$ by Lemma \ref{semiab}.
\end{remark}

Now we recall several celebrated results regarding the computation of $(GL_n(R))^{ab}$. 

\begin{theorem}\cite[Proposition~9.2]{Cohn}\label{Cohn} If $R$ is an ED satisfying $1=u+v$ for units $u,v$, then $[GL_2(R),GL_2(R)] = E_2(R)$.
\end{theorem}

\begin{theorem}\cite[Corollary~2.3]{Heit}\cite[Theorem~10.15]{Magurn}\label{SLn}
Let $R$ be an unital ring with Krull dimension at most $1$, for example, a PID, then $[GL_n(R),GL_n(R)]=E_n(R)$ for $n>2$. 
\end{theorem}

\begin{theorem}\cite[Theorem~4.3.9]{HO}\label{ED}
If $R$ is an ED, then $SL_n(R) = E_n(R)$ for $ n \geq 1$.
\end{theorem}

Suppose the theory $T$ of the module $M_R$ satisfies $T=T^{\aleph_0}$. Under the combined hypotheses of Lemma \ref{quotfree}, and Theorems \ref{Cohn}, \ref{SLn} and \ref{ED} on $R$, we can follow through the proof of Proposition \ref{Omega} to obtain 
\begin{equation}\label{omeganab}
(\Omega_n^n)^{ab}\cong (GL_n(R))^{ab}\oplus\bigoplus_{i=0}^{n-1}((GL_i(R))^{ab}\oplus\mathbb Z_2).
\end{equation} Then we use the directed colimit description of $K_1(M_R)$ in Theorem \ref{mainth} to get a nice explicit expression for $K_1(M_R)$.
\begin{theorem}\label{T=T0}
Let $R$ be a Euclidean domain satisfying $1=u+v$ for some units $u,v$, and $M_R$ be an infinite free right $R$-module. If the theory $T$ of $M_R$ satisfies $T = T^{\aleph_0}$ then  $$K_1(M_R) \cong K_1(R_R) \cong \bigoplus_{n=0}^\infty((GL_n(R))^{ab}\oplus\mathbb Z_2)\cong\mathbb Z_2\oplus\bigoplus_{n=1}^\infty(R^\times\oplus\mathbb Z_2).$$ 
\end{theorem}
Now suppose that the theory $T$ of the module $M_R$ does not satisfy $T=T^{\aleph_0}$. Then in addition to all the results used in the proof of the above theorem, we also need to use the dichotomy in the following lemma while obtaining the analogue of Lemma \ref{Omega}.
\begin{lemma}\label{evind}
Suppose the theory $T$ of the module $M_R$ satisfies $T\neq T^{\aleph_0}$. Let $(\mathbf G_1,\leq)^{op}$ be the directed system described in Proposition \ref{Upsi}. If this directed system contains a cofinal system of even-indexed subgroups of $M$ then for $n\geq1$, $(\Upsilon^n)^{ab} \cong (GL_n(R)\ltimes M^n)^{ab}$; otherwise, $(\Upsilon^n)^{ab} \cong (GL_n(R)\ltimes M^n)^{ab} \oplus \mathbb{Z}_2$.     
\end{lemma}
\begin{proof}
This follows from Proposition \ref{Upsi} and Proposition \ref{evsbg} since abelianization commutes with colimits.
\end{proof}
Using the notations of the above lemma, if $(\mathbf G_1,\leq)^{op}$ contains a cofinal system of even-indexed subgroups then the isomorphism in Equation \eqref{omeganab} still remains valid; otherwise for $n\geq 1$ we have
$$(\Omega^n_n)^{ab} \cong ((GL_n(R))^{ab}\oplus \mathbb{Z}_2)\oplus \bigoplus_{i=1}^{n-1}((GL_i(R))^{ab} \oplus \mathbb{Z}_2 \oplus \mathbb{Z}_2) \oplus \mathbb{Z}_2.$$
Using these expressions together with the directed colimit description of $K_1(M_R)$ in Theorem \ref{mainth}, we get the following result. 
\begin{theorem}\label{TneqT0}
Let $R$ be a Euclidean domain satisfying $1=u+v$ for some units $u,v$, and $M_R$ be an infinite free right $R$-module. Suppose that the theory $T$ of $M_R$ satisfies $T \neq T^{\aleph_0}$. Using the notations of Lemma \ref{evind}, if $(\mathbf G_1,\leq)^{op}$ contains a cofinal system of even-indexed subgroups of $M$ then $$K_1(M_R)\cong K_1(R_R) \cong \bigoplus_{n=0}^\infty((GL_n(R))^{ab}\oplus\mathbb Z_2)\cong\mathbb Z_2\oplus\bigoplus_{n=1}^\infty(R^\times\oplus\mathbb Z_2);$$ otherwise 
$$K_1(M_R) \cong K_1(R_R) \cong \mathbb Z_2 \oplus \bigoplus_{n=1}^\infty((GL_n(R))^{ab}\oplus\mathbb Z_2\oplus \mathbb{Z}_2) \cong \mathbb Z_2 \oplus \bigoplus_{n=1}^\infty(R^\times\oplus\mathbb Z_2\oplus \mathbb{Z}_2).$$
\end{theorem}

\begin{corollary}\label{polyring}
If $F$ is a field with $\mathrm{char}(F)=0$ and $M_{F[X]}$ is a non-zero free module over $F[x]$ then $$K_1(M_{F[X]}) \cong K_1(F[X]_{F[X]})\cong K_1(F_F) \cong\mathbb Z_2\oplus\bigoplus_{n=1}^\infty(F^\times\oplus\mathbb Z_2).$$ 
\end{corollary}

If $F$ is a finite field $\mathrm{char}(F)=2$ and $F$ contains at least $4$ elements, then the field $F_4$ with $4$ elements embeds into $F$. Moreover, if $F_4=\{0,1,a,b\}$ then $a+b=1$, and thus $F$ satisfies the hypotheses of the first case of Theorem \ref{TneqT0}, and we get the following.
\begin{corollary}\label{F2k}
Let $F$ be the finite field $F_{2^k}$, where $k\geq 2$. Then for every infinite $F$-vector space $V_F$, we have
$$K_1(V_F)\cong K_1(F_F) \cong \bigoplus_{n=0}^\infty((GL_n(F))^{ab}\oplus\mathbb Z_2)\cong\mathbb Z_2\oplus\bigoplus_{n=1}^\infty(\mathbb Z_{2^k-1}\oplus\mathbb Z_2).$$    
\end{corollary}

If $F$ is a finite field  $F_{p^k}$, where $k\geq 1$ and $p$ is a prime greater than $2$. Then $-1$ and $2$ are units in $F$ satisfying $1=2+(-1)$ in $F$. Moreover, $F$ satisfies the hypotheses of the second case of Theorem \ref{TneqT0}, and thus we get the following.
\begin{corollary}\label{Fpk}
    Let $F$ be the finite field $F_{p^k}$, where $k\geq 1$ and $p$ is a prime greater than $2$. Then for every infinite $F$-vector space $V_F$, we have
$$K_1(V_F) \cong K_1(F_F) \cong \mathbb Z_2 \oplus \bigoplus_{n=1}^\infty((GL_n(F))^{ab}\oplus\mathbb Z_2\oplus \mathbb{Z}_2) \cong \mathbb Z_2 \oplus \bigoplus_{n=1}^\infty(\Z_{p^k-1}\oplus\mathbb Z_2\oplus \mathbb{Z}_2).$$
\end{corollary}

\begin{remark}\label{failF2}
Suppose $F$ is the field $F_2$ with $2$ elements and $V_F$ is an infinite vector space. Since $(GL_1(F))^{ab}$ is the trivial group, we obtain $(\Upsilon^1(V_F))^{ab}\cong V$, which creates obstacles in the computation of $(\Omega^1_1)^{ab}$ using the methods developed so far. Thus, we are unable to use the same recipe to compute $K_1(V_F)$. 
\end{remark}

\section{Computation of $K_1(\mathbb{Z}_\mathbb Z)$}\label{DefZ}
Recall that the theory $T:=Th(\mathbb Z_\mathbb Z)$ of the abelian group $\mathbb Z_\mathbb Z$ of integers is not closed under products. Even though the ring $\mathbb Z$ of integers is an ED, the identity $1=u+v$ fails to hold for any units $u,v$ in $\mathbb Z$. Hence Theorem \ref{TneqT0} is not applicable for the computation of $K_1(\mathbb Z_{\mathbb Z})$; however, we use a slight modification to compute $K_1(\mathbb Z_\mathbb Z)$ in Theorem \ref{ZK[x]}.

We have $$(GL_n(\mathbb Z))^{ab}\cong\begin{cases}\mathbb Z_2&\text{if }n=1\text{ or }n\geq3;\\\Z_2\oplus\Z_2&\text{if }n=2.\end{cases}$$ Since abelianization commutes with directed colimits, Propositions \ref{Aut}, \ref{Upsi} and \ref{evsbg}, Lemmas  \ref{semiab}, \ref{wt} and \ref{quotfree}, and Remark \ref{chad} together give that $$(\Upsilon^n(\Z_\Z))^{ab}\cong(\mathrm{Aut}_\mathcal L(\mathbb Z^n))^{ab}\cong(GL_n(\mathbb Z)\ltimes\mathbb Z^n)^{ab}\cong\begin{cases}(GL_n(\Z))^{ab}&\text{if }n\geq2;\\(GL_1(\Z))^{ab}\oplus\Z_2&\text{if }n=1.\end{cases}$$

Since $\Omega^1_1\cong\Upsilon^1\ltimes\Sigma^1_0$, we have $$(\Omega^1_1(\mathbb{Z}))^{ab}\cong(\Upsilon^1)^{ab}\oplus(\Sigma^1_0)^{ab}\cong((GL_1(\Z))^{ab}\oplus\mathbb{Z}_2) \oplus (\Sigma^1_0)^{ab}\cong(\Z_2\oplus\Z_2)\oplus\mathbb{Z}_2.$$

Since $\Omega^2_2\cong \Upsilon^2\ltimes \Omega^2_1\cong\Upsilon^2\ltimes((\Upsilon^1\wr\Sigma^2_1)\ltimes\Sigma^2_0)$, we get
\begin{align*}
(\Omega^2_2)^{ab} &\cong (\Upsilon^2)^{ab}\oplus((\Omega^2_1)^{ab})_{\Upsilon^2}\\&\cong (\Upsilon^2)^{ab}\oplus(((\Upsilon^1)^{ab}\oplus(\Sigma^2_1)^{ab})\oplus((\Sigma^2_0)^{ab})_{(\Upsilon^1\wr\Sigma^2_1)})_{\Upsilon^2}\\&\cong(GL_2(\Z))^{ab}\oplus((((GL_1(\Z))^{ab}\oplus\mathbb{Z}_2)\oplus(\Sigma^2_1)^{ab})\oplus(\Sigma^2_0)^{ab})_{\Upsilon^2}\\&\cong (\mathbb{Z}_2 \oplus \mathbb{Z}_2) \oplus (((\mathbb{Z}_2 \oplus \mathbb{Z}_2) \oplus \mathbb{Z}_2 ) \oplus \mathbb{Z}_2)_{\Upsilon^2}.
\end{align*}
Recall that $\Upsilon^2$ acts on $\Omega^2_1$ by conjugation. Also recall from Lemma \ref{semiab} that $((\Omega^2_1)^{ab})_{\Upsilon^2}\cong(\Omega^2_1)^{ab}/\{h^gh^{-1}\mid h\in(\Omega^2_1)^{ab},g\in\Upsilon^2\}$. Now focus on the only copy of $\Z_2$ that appears in the second last line in the above sequence of isomorphisms. Apart from this copy of $\Z_2$, $\Upsilon^2$ acts trivially on all other direct summands of $(\Omega^2_1)^{ab}$; however, we cannot determine whether this action is trivial or not on this copy. Therefore, we get that $(\Omega^2_2)^{ab}\cong\Z^{k_2}$ for $k_2=5$ or $6$. The computations of $(\Omega_n^n)^{ab}$ for $n>2$ proceed similarly so that we obtain $(\Omega_n^n)^{ab}\cong\Z^{k_n}$ for some $k_n$ satisfying $k_n<k_{n+1}$ for each $n\geq1$. Finally, using Theorem \ref{mainth} we get the following.
\begin{theorem}\label{ZK[x]}
We have $K_1(\mathbb{Z}_{\mathbb{Z}}) \cong \bigoplus_{n=0}^{\infty} \mathbb{Z}_2.$ 
\end{theorem}

\begin{remark}\label{failMz}
We cannot compute $K_1(M_{\mathbb Z})$ for an infinite free $\mathbb Z$-module $M_{\mathbb Z}$ in a similar way for a non-trivial quotient of $M_\Z$ appears in the expression of $(\Upsilon^1(M_\Z))^{ab}$, which will create obstacles in the computation of $(\Omega^2_2)^{ab}$.
\end{remark}

\section{Connection with algebraic $K_1$}\label{algK1modK1}
Fix a PID $R$. Let us recall the definition of algebraic $K_1$ of $R$. Let $P(R)$ be the skeletally small groupoid of finitely generated projective $R$-modules under $R$-module isomorphisms. This groupoid is equipped with a symmetric monoidal structure of the direct sum ($\oplus$). The \emph{algebraic $K_1$} of a ring $R$, denoted $K_1^{\oplus}(R)$, is defined to be $K_1^\oplus(P(R))$. Since $R$ is a PID, every finitely generated projective $R$-module is free. Thus $P(R)$ is equivalent to the symmetric monoidal groupoid $\mathrm{Free}(R)$ whose objects are finitely generated free $R$-modules. Translations are faithful in both $P(R)$ and $\mathrm{Free}(R)$, and thus Bass' description (Theorem \ref{K1Bass}) can be used for the computation of $K_1^\oplus(R)$.

We established in the proof of Proposition \ref{PIDcolour} that every $pp$-definable set in $\mathcal L(R_R)$ is in bijection with a finitely generated free $R$-module. Also note that the directed colimit in Theorem \ref{mainth} is actually directly from Bass' description of $K_1$ (Theorem \ref{K1Bass}). The following theorem establishes a connection between algebraic and model-theoretic $K_1$.
\begin{theorem}\label{algmodK1}
If $R$ is a PID then there is a natural embedding of $K_1^\oplus(R)$ into $K_1(R_R)$.
\end{theorem}

\begin{proof}
Since every $R$-linear automorphism of $R^n$ is $pp$-definable, the object assignment $R^n\mapsto R_R^n:\mathrm{Free}(R)\to\C(R_R)$ together with morphism assignment $f\mapsto f$ defines a faithful functor. Moreover, the map $\overline{x} \mapsto(\overline{x},0)$ is both an $R$-module embedding $R^n\hookrightarrow R^{n+1}$ and a definable injection $R_R^n\hookrightarrow R_R^{n+1}$. Thus defining $i:\mathrm{Aut}_{\mathrm{Free}(R)}(R^n)\to\mathrm{Aut}_{\C(R_R)}(R_R^n)$ by $i(f) := (f \oplus id_R)$, we get the following commutative diagram of groups.
$$
\xymatrix{
\mathrm{Aut}_{\mathrm{Free}(R)}(R^n) \ar@{^{(}->}[r]\ar[d]^{i} &\mathrm{Aut}_{\C(R_R)}(R^n_R) \ar[d]^{i}\\
\mathrm{Aut}_{\mathrm{Free}(R)}(R^{n+1}) \ar@{^{(}->}[r] &\mathrm{Aut}_{\C(R_R)}(R^{n+1}_R)}$$

Finally, since the abelianization functor commutes with colimits, Bass' description of $K_1$ given in Theorem \ref{K1Bass} allows us to conclude that there is a natural embedding $K_1^\oplus(R)\to K_1(R_R)$.
\end{proof}

The algebraic $K_1$ of Euclidean domains is well-understood.
\begin{theorem}\cite[Example~1.3.5]{Weibel}\label{algK1ED}
If $R$ is an ED, then $K_1^\oplus(R)\cong R^\times$, where the abelianization map takes $A\in GL(R):=\varinjlim_{n\in\mathbb N} GL_n(R)$ to $\det(A)\in K_1(R)$.
\end{theorem}
 
From the explicit computations of $K_1(R_R)$ in this paper for different EDs, it can be seen that the model-theoretic $K_1$ is much larger compared to the algebraic $K_1$--this is not a surprise as there are far too many definable self-bijections in comparison with linear automorphisms.

\begin{remark}
Suppose $R$ is an ED satisfying $1=u+v$ for some $u,v \in R^\times$. Then the composition $GL(R) \twoheadrightarrow K_1^\oplus(R)\hookrightarrow K_1(R_R)$ can be described as follows: if $A \in GL_n(R)$ is not in the image of the natural embedding of $GL_{n-1}(R)$ into $GL_n(R)$, then it maps to $\det(A)\in(GL_n(R))^{ab}$, where $(GL_n(R))^{ab}$ is the leading term of $(\Omega^n_n)^{ab}$. 
\end{remark}
\section*{Acknowledgements} A part of this work, \S~\ref{SemGrP}-\ref{K1V}, appeared in Chapter~3 of the PhD thesis \cite{kuberthesis} of the second author. This part of the work was supported by an Overseas Students Fellowship of the School of Mathematics, University of Manchester.

\printbibliography
\end{document}